\documentclass{amsart}%
\usepackage{amsfonts}
\usepackage{amsmath}
\usepackage{amssymb}
\usepackage{graphicx}%
\providecommand{\U}[1]{\protect\rule{.1in}{.1in}}
\newtheorem{theorem}{Theorem}
\theoremstyle{plain}
\newtheorem{acknowledgement}{Acknowledgement}

\newtheorem{corollary}{Corollary}

\newtheorem{example}{Example}

\newtheorem{proposition}{Proposition}
\newtheorem{remark}{Remark}

\numberwithin{equation}{section}
\begin{document}
\title[Cost Convexity]{Structural results on convexity relative to cost functions}
\author{Flavia-Corina Mitroi}
\address[Flavia-Corina Mitroi]{University of Craiova, Department of Mathematics, Street
A. I. Cuza 13, Craiova, RO-200585, Romania}
\email{fcmitroi@yahoo.com}
\author{Daniel Alexandru Ion}
\address[Daniel Alexandru Ion]{University of Craiova, Department of Mathematics, Street
A. I. Cuza 13, Craiova, RO-200585, Romania }
\email{dan\_alexion@yahoo.com}

\begin{abstract}
Mass transportation problems appear in various areas of mathematics, their
solutions involving cost convex potentials. Fenchel duality also represents an
important concept for a wide variety of optimization problems, both from the
theoretical and the computational viewpoints. We drew a parallel to the
classical theory of convex functions by investigating the cost convexity and
its connections with the usual convexity.\ We give a generalization of
Jensen's inequality for $c$-convex functions.

\end{abstract}
\maketitle

\section{Introduction}

Let $I$ and $J$ be two bounded intervals. Assume $f\ $is a real valued
function defined on $I$ such that there exists $g$ a real valued function
defined on $J\ $which satisfies
\begin{equation}
f(x)=\sup_{y\in J}\left\{  xy-g\left(  y\right)  \right\}  . \label{sup}%
\end{equation}
The function $f$ is called the Fenchel transform (conjugate) of $g.$ It is
known that (\ref{sup}) characterizes convex functions (see \cite{F1949}).

Throughout this paper the cost function $c:I\times J\rightarrow%
\mathbb{R}
$ is continuous (unless otherwise indicated); it represents the cost per unit
mass for transporting material from $x$ $\in I$ to $y\in J$.

A proper function $f:I\rightarrow\left(  -\infty,\infty\right]  $ is said to
be $c$-convex (see for instance \cite{caf1996},\cite{R2007},\cite{vil09}) if
there exists $g:J\rightarrow\left(  -\infty,\infty\right]  $ such that for all
$x\in I$ we have%
\[
f\left(  x\right)  =\sup_{y\in J}\left\{  c\left(  x,y\right)  -g\left(
y\right)  \right\}  .
\]
It adapts the notion of a convex function to the geometry of the cost
function. Its $c$-transform ($c$-conjugate) is $f^{c}$ defined by
\[
f^{c}\left(  y\right)  =\sup_{x\in I}\left\{  c\left(  x,y\right)  -f\left(
x\right)  \right\}  .
\]
If for a fixed $x_{0}$ the supremum is obtained at $y_{0},$ then we say that
$c\left(  x,y_{0}\right)  -g\left(  y_{0}\right)  $ supports $f$ (is tangent
by below \cite{caf1996}) at $x_{0}.$ One has the double $c$-conjugate
\[
f^{cc}\left(  x\right)  =\sup_{y\in J}\inf_{z\in I}\left\{  f\left(  z\right)
+c\left(  x,y\right)  -c\left(  z,y\right)  \right\}
\]
for all $x\in I.$ This is the largest $c$-convex function majorized by $f,$
that is $f^{cc}\leq f$ (see \cite[pp. 125]{RR1988}). We also recall that the
condition $f=f^{cc}$ is equivalent to the $c$-convexity of\ $f$ (see
\cite[Proposition 5.8]{vil09})$.$

Replacing the supremum by the infimum one gets the definition of cost concavity.

Before stating the results we establish the notation and recall some
definitions from the literature (see \cite{RR1988}).

Given a function $f:I\rightarrow%
\mathbb{R}
$, we say that $f$ admits a $c-$support curve at $x_{0}\in I$ if there exists
$y\in J$ such that%
\[
f(x)\geq f(x_{0})+c\left(  x,y\right)  -c\left(  x_{0},y\right)  ,\text{ for
all }x\in I.
\]

The $c-$subdifferential\emph{\ }($c-$normal mapping \cite{TW09})\emph{ }of a
real function\emph{\ }$f$ defined on an interval $I$ is a multivalued function
$\partial_{c}f:I\rightarrow\mathcal{P}(J)$ given by%
\[
\partial_{c}f(x_{0})=\left\{  y\in J:f(x)\geq f(x_{0})+c\left(  x,y\right)
-c\left(  x_{0},y\right)  \text{, for every}\,\,x\in I\right\}  .
\]
The elements of $\partial_{c}f(x)$ are called $c$-subgradients at $x$.

We denote throughout the paper the effective domain of the $c-$subdifferential
by
\[
dom\left(  \partial_{c}f\right)  =\left\{  x_{0}\in I:\partial_{c}f(x_{0}%
)\neq\varnothing\right\}  .
\]
Every $c$-convex function admits a $c-$support curve at\ each interior point
of its domain, that is $f$ satisfies $dom\left(  \partial_{c}f\right)
\supseteq int(I).$ Clearly, if the cost function is differentiable in its
first variable at $x_{0},$ then we also have
\[
\frac{\partial c}{\partial x}\left(  x_{0},\partial_{c}f(x_{0})\right)
\subseteq\partial f(x_{0}).
\]
The map $x\rightarrow c\left(  x,y\right)  -f(x)$ is maximized at $x_{0}%
,$\ and so we have $y\in\partial_{c}f(x_{0})$ if and only if $f^{c}\left(
y\right)  =c\left(  x_{0},y\right)  -f(x_{0}).$ It follows that a $c-$convex
function $f$ can be represented as
\begin{equation}
f\left(  x\right)  =\sup_{y\in\partial_{c}f(x)}\left\{  c\left(  x,y\right)
-g\left(  y\right)  \right\}  \label{c_conv}%
\end{equation}
for every$\,\,x\in dom\left(  \partial_{c}f\right)  .$ Obviously then
\[
f^{cc}\left(  x\right)  =\sup_{y\in\partial_{c}f(x)}\inf_{z\in I}\left\{
f\left(  z\right)  +c\left(  x,y\right)  -c\left(  z,y\right)  \right\}
\]
for all $x\in I.$

Similar concepts were developed for $c$-concave functions in \cite{TW2009}.
Some authors (see for instance \cite{MTW2005},\cite[Section 6]{TW09}) consider
by definition that a function $f$ is $c-$concave if $dom\left(  \partial
_{c}f\right)  =I,$ that is if it admits $c-$support curve at any point of its
domain. For this they assume the function $f$ to be upper semicontinuous.

For the particular case $c\left(  x,y\right)  =xy$ we get from (\ref{c_conv})
the usual convexity of $f$. Obviously then we recover the definitions of the
usual subdifferential $\partial f$ and of the support lines for convex
functions. For the usual convex functions we will use the well-known notation
$f^{c}=f^{\ast}$ and $f^{cc}=f^{\ast\ast}.$

The aim of this paper is to investigate the cost convexity and to establish
some connections with the usual convexity. See also \cite{MN2011} for more
results on this topic. Before stating the results, since much of our attention
here will be devoted to Jensen's inequality (see \cite{cpn06}), we recall for
the reader's convenience its classical statement, both the discrete and
integral forms:

J1) \emph{Let }$x_{i}\in I,$\emph{ }$p_{i}>0,$\emph{ }$i=1,...,n,$\emph{
}$\sum p_{i}=1.$\emph{ Then}%
\[
f\left(  \sum p_{i}x_{i}\right)  \leq\sum p_{i}f\left(  x_{i}\right)
\]
\emph{holds for every convex function }$f:I\rightarrow%
\mathbb{R}
.$

J2) \emph{Let }$h:[a,b]\rightarrow I$\emph{ be an integrable function. Then
\ }%
\[
f\left(  \frac{1}{b-a}\int_{a}^{b}h\left(  x\right)  \mathrm{d}x\right)
\leq\frac{1}{b-a}\int_{a}^{b}f\left(  h\left(  x\right)  \right)  \mathrm{d}x
\]
\emph{holds for every convex function }$f:I\rightarrow%
\mathbb{R}
$\emph{, provided }$f\circ h$\emph{ is integrable}$.$

\section{Main results}

\subsection{Jensen's inequality for $c-$convex functions}

In what follows $i$-affine (convex, concave) stands for "affine (convex,
concave) in the $i$-th variable". We firstly state and prove the discrete and
continuous forms of Jensen's inequality for $c-$convex functions.

\begin{theorem}
[the discrete form of Jensen's inequality]\label{JC_1} Let $c:I\times
J\rightarrow%
\mathbb{R}
$ be a cost function. {Assume }$f:I\rightarrow\mathbb{%
\mathbb{R}
}$ is a{\ }$c-${convex func}tion. Let $n\geq2,$ $x_{i}\in I,$ $p_{i}>0,$
$i=1,...,n,$ $\sum p_{i}=1$. Let $y\in\partial_{c}f\left(  \sum p_{i}%
x_{i}\right)  .$ Then%
\[
\sum p_{i}f\left(  x_{i}\right)  -f\left(  \sum p_{i}x_{i}\right)  \geq\sum
p_{i}c\left(  x_{i},y\right)  -c\left(  \sum p_{i}x_{i},y\right)  .
\]

\end{theorem}

\begin{proof}
We consider the $c-$support curve at $\sum p_{i}x_{i}\ $corresponding to the
$c-$gradient $y.$ It holds%
\[
f(x)\geq f\left(  \sum p_{i}x_{i}\right)  +c\left(  x,y\right)  -c\left(  \sum
p_{i}x_{i},y\right)  ,
\]
for all $x\in I.$ Particularly we can write
\[
f(x_{i})\geq f\left(  \sum p_{i}x_{i}\right)  +c\left(  x_{i},y\right)
-c\left(  \sum p_{i}x_{i},y\right)  ,
\]
for $i=1,...,n$. By multiplying both sides by $p_{i}\ $and summing over $i$ we
get the claimed result.
\end{proof}

\begin{corollary}
Let $c:I\times J\rightarrow%
\mathbb{R}
$ be a cost function and{ }$f:\left[  a,b\right]  \rightarrow\mathbb{%
\mathbb{R}
}$ be $c-${convex}. Then{%
\begin{equation}
\frac{f\left(  a\right)  +f\left(  b\right)  }{2}-f\left(  \frac{a+b}%
{2}\right)  \geq\frac{c\left(  a,y\right)  +c\left(  b,y\right)  }{2}-c\left(
\frac{a+b}{2},y\right)  \label{discrete_j}%
\end{equation}
}for all $y\in\partial_{c}f\left(  \frac{a+b}{2}\right)  $.
\end{corollary}

\begin{proof}
We apply Theorem \ref{JC_1}, taking $x_{1}=a,$ $x_{2}=b,$ $p_{1}=p_{2}%
=\frac{1}{2}.$ Then $\frac{a+b}{2}\in\left(  a,b\right)  \subseteq dom\left(
\partial_{c}f\right)  .$
\end{proof}

For $c\left(  x,y\right)  =xy$ we recapture the inequality
\[
f\left(  \frac{a+b}{2}\right)  \leq\frac{f\left(  a\right)  +f\left(
b\right)  }{2}.
\]
Another straightforward consequence of Theorem \ref{JC_1} reads as follows.

\begin{corollary}
Let $c:I\times J\rightarrow%
\mathbb{R}
$ be a cost function and{ }$f:\left[  a,b\right]  \rightarrow\mathbb{%
\mathbb{R}
}$ be $c-${convex}. Let $y\in\partial_{c}f\left(  \frac{a+b}{2}\right)  $ and
$g:[a,b]\rightarrow\mathbb{%
\mathbb{R}
},$ $g\left(  x\right)  =c\left(  x,y\right)  -f\left(  x\right)  .$\ Then
\[
g\left(  \frac{a+b}{2}\right)  \geq\frac{g\left(  a\right)  +g\left(
b\right)  }{2}.
\]

\end{corollary}

\begin{proof}
Directly from (\ref{discrete_j}).
\end{proof}

Under $c-$convexity conditions, the integral Jensen's inequality is given by
the following theorem.

\begin{theorem}
[the integral form of Jensen's inequality]\label{JI}Let $c:I\times
J\rightarrow%
\mathbb{R}
$ be a cost function and{ }$f:\left[  a,b\right]  \rightarrow\mathbb{%
\mathbb{R}
}$ be continuous and $c-${convex}. Then {\
\begin{equation}
f\left(  \frac{a+b}{2}\right)  \left(  b-a\right)  +\int_{a}^{b}\left[
c\left(  x,y\right)  -c\left(  \frac{a+b}{2},y\right)  \right]  \mathrm{d}%
x\leq\int_{a}^{b}f\left(  x\right)  \mathrm{d}x \label{J_1}%
\end{equation}
}for all $y\in\partial_{c}f\left(  \frac{a+b}{2}\right)  $.
\end{theorem}

\begin{proof}
Let $y\in\partial_{c}f\left(  \frac{a+b}{2}\right)  .$\ We consider the
$c-$support curve at $\frac{a+b}{2}\ $corresponding to the $c-$gradient $y.$
It holds%
\[
f(x)\geq f\left(  \frac{a+b}{2}\right)  +c\left(  x,y\right)  -c\left(
\frac{a+b}{2},y\right)
\]
for all $x\in I.$ To complete the proof, it remains to integrate the
inequality on $\left[  a,b\right]  .$
\end{proof}

One can use the same recipe in order to obtain the weighted form of integral
Jensen's inequality, replacing the Lebesgue measure by a Borel probabilistic
measure $\mu$ on $\left[  a,b\right]  $ with the barycenter $b_{\mu}\in\left(
a,b\right)  .$ Thus {%
\[
f\left(  b_{\mu}\right)  +\int_{a}^{b}\left[  c\left(  x,y\right)  -c\left(
b_{\mu},y\right)  \right]  \mathrm{d}\mu\left(  x\right)  \leq\int_{a}%
^{b}f\left(  x\right)  \mathrm{d}\mu\left(  x\right)  ,
\]
}for all $y\in\partial_{c}f\left(  b_{\mu}\right)  $.

\begin{remark}
Obviously (\ref{J_1}) can be written in a more general form using another
point $\xi\in dom\left(  \partial_{c}f\right)  $ instead of $\frac{a+b}{2}.$
Then {%
\begin{equation}
f\left(  \xi\right)  \left(  b-a\right)  +\int_{a}^{b}\left[  c\left(
x,y\right)  -c\left(  \xi,y\right)  \right]  \mathrm{d}x\leq\int_{a}%
^{b}f\left(  x\right)  \mathrm{d}x, \label{cpn}%
\end{equation}
where }$y\in\partial_{c}f\left(  \xi\right)  .$
\end{remark}

From (\ref{cpn}), for the particular case $c\left(  x,y\right)  =xy$ we
recapture a result due to C.P. Niculescu and L.E. Persson \cite[p. 668]{cpn04}:

\begin{corollary}
{Let }$f:\left[  a,b\right]  \rightarrow\mathbb{%
\mathbb{R}
}$ be a{ continuous, convex func}tion, $\xi\in\left(  a,b\right)  $. It holds{%
\[
f\left(  \xi\right)  +y\left(  \frac{a+b}{2}-\xi\right)  \leq\frac{1}{b-a}%
\int_{a}^{b}f\left(  x\right)  \mathrm{d}x,
\]
where }$y\in\partial f\left(  \xi\right)  .$
\end{corollary}

\begin{corollary}
\label{J_cor1}{All continuous }functions $f:\left[  a,b\right]  \rightarrow
\mathbb{%
\mathbb{R}
},$ which are{ }convex relative to 1-affine costs, satisfy {%
\[
f\left(  \frac{a+b}{2}\right)  \leq\frac{1}{b-a}\int_{a}^{b}f\left(  x\right)
\mathrm{d}x.
\]
}
\end{corollary}

\begin{proof}
Since the function $c\left(  x,y\right)  $ is\ 1-affine,
\[
c\left(  \frac{a+b}{2},y\right)  \left(  b-a\right)  =\int_{a}^{b}c\left(
x,y\right)  \mathrm{d}x.
\]
We use (\ref{J_1}). This completes the proof.
\end{proof}

The 1-affine cost functions can be expressed as $c\left(  x,y\right)
=a\left(  y\right)  x+b\left(  y\right)  $ with $a,b:J\rightarrow%
\mathbb{R}
.$ The cost function $c\left(  x,y\right)  =xy$ is obviously 1-affine and
Corollary \ref{J_cor1} applies, hence the known Jensen's inequality for convex
functions becomes a particular case of Theorem \ref{JI}. In the light of
Jensen's inequality it appears that the convexity relative to 1-affine cost
functions implies the usual convexity.

\subsection{The $c-$convexity and the role of the $c-$subdifferential}

We establish next some new connections between the usual convexity and the
cost convexity. Due to its dependence on the cost function, the concept of
cost subdifferential is providing conceptual clarity and plays a crucial role
in what follows.

Every continuous $c-$convex function is the upper envelope of its $c$-support
curves. More precisely:

\begin{proposition}
\label{env}Let $c:I\times J\rightarrow%
\mathbb{R}
$ be uniformly continous and{ }$f:I\rightarrow\mathbb{%
\mathbb{R}
}$ be continuous and $c-${convex}. Assume $y$ is a selection of $\partial
_{c}f$, that is $y\left(  t\right)  \in\partial_{c}f(t)$ for all $t\in
dom\left(  \partial_{c}f\right)  .$ Then%
\[
f\left(  x\right)  =\sup_{t\in int(I)}\left\{  f(t)+c\left(  x,y\left(
t\right)  \right)  -c\left(  t,y\left(  t\right)  \right)  \right\}
\]
for all $x\in I.$
\end{proposition}

\begin{proof}
The case of interior points is clear. Let $x\ $be an endpoint, say the
leftmost one. By the continuity at $x$, for each $\varepsilon>0$ there exists
$\delta_{\varepsilon}>0$ such that for all $t$ with $\left\vert t-x\right\vert
<\delta_{\varepsilon}$ we have $\left\vert f(t)-f(x)\right\vert <\frac
{\varepsilon}{2}$ and $\left\vert c\left(  t,y\left(  t\right)  \right)
-c\left(  x,y\left(  t\right)  \right)  \right\vert <\frac{\varepsilon}{2}.$
This shows that
\[
f(x)+\varepsilon>f(t)+c\left(  x,y\left(  t\right)  \right)  -c\left(
t,y\left(  t\right)  \right)
\]
for $t\in\left(  x,x+\delta_{\varepsilon}\right)  .$ We also have
\[
\lim_{t\rightarrow x+}\left[  c\left(  x,y\left(  t\right)  \right)  -c\left(
t,y\left(  t\right)  \right)  \right]  =0
\]
and the result follows.
\end{proof}

In the context of usual convexity, Proposition \ref{env} has the following
known corollary:

\begin{corollary}
[{\cite[Theorem 1.5.2]{cpn06}}]Let $f:I\rightarrow%
\mathbb{R}
$ be continuous and convex. Assume $y$ is a selection of $\partial f$, that is
$y\left(  t\right)  \in\partial f(t)$ for all $t\in I.$ Then%
\[
f\left(  x\right)  =\sup_{t\in int(I)}\left\{  f(t)+\left(  x-t\right)
y\left(  t\right)  \right\}
\]
for all $x\in I.$
\end{corollary}

The following proposition lets us see the way the $c-$subdifferential and the
subdifferential are connected.

\begin{proposition}
[relating $c-$subdifferentials to subdifferentials]\label{c_sd}Let $c:I\times
J\rightarrow%
\mathbb{R}
$ be a cost function and{ }$f:I\rightarrow\mathbb{%
\mathbb{R}
}$. It holds
\[
\left(  x,y\right)  \in\partial_{c}f\Rightarrow\partial c_{y}(x)\subseteq
\partial f\left(  x\right)  ,
\]
where $c_{y}(x)=c(x,y).$ Moreover if $f$ is differentiable and $c$ is
differentiable in its first variable, then $\frac{\partial c}{\partial
x}\left(  x,y\right)  =f^{\prime}\left(  x\right)  .$
\end{proposition}

\begin{proof}
For $\left(  x,y\right)  \in\partial_{c}f,$ $\alpha\in\partial c_{y}(x)$ we
have%
\[
f(z)-f(x)\geq c\left(  z,y\right)  -c\left(  x,y\right)  \geq\alpha\left(
z-x\right)  .
\]
for all $z\in I.$\ It leads to $\alpha\in\partial f\left(  x\right)  .$ Under
the differentiability assumptions we also have $\partial c_{y}(x)=\left\{
\frac{\partial c}{\partial x}\left(  x,y\right)  \right\}  $ and $\partial
f\left(  x\right)  =\left\{  f^{\prime}\left(  x\right)  \right\}  .$

The proof is completed.
\end{proof}

The counterpart of Proposition \ref{c_sd}, for $c$-superdifferentials, can be
read in \cite[Lemma 3.1, Lemma C.7]{GM1996}, for the particular case
$c=h\left(  x-y\right)  $.

\begin{proposition}
Let $c:I\times J\rightarrow%
\mathbb{R}
$ be a cost function. For all $x\in I$ and $y\in J$ we have $\left(
x,y\right)  \in\partial_{c}c_{y}.$
\end{proposition}

\begin{proof}
The proof is an immediate consequence of the definition of the $c$ - subdifferential.
\end{proof}

\begin{proposition}
Suppose that $c:I\times J\rightarrow%
\mathbb{R}
$ is a cost function and{ }$f,g:I\rightarrow%
\mathbb{R}
\ $are{\ }$c-${convex}. It holds
\[
\partial_{c}f(x)\cap\partial_{c}g(x)\subset\partial_{c}\left(  \left(
1-\lambda\right)  f+\lambda g\right)  (x)
\]
for all $\lambda\in\left[  0,1\right]  .$
\end{proposition}

\begin{proof}
Assume $\partial_{c}f(x)\cap\partial_{c}g(x)\neq\varnothing.$ Let
$y\in\partial_{c}f(x)\cap\partial_{c}g(x)$. Then%
\begin{align*}
f(y)  &  \geq f(x)+c\left(  z,y\right)  -c\left(  x,y\right)  ,\\
g(y)  &  \geq g(x)+c\left(  z,y\right)  -c\left(  x,y\right)  ,
\end{align*}
for all $z\in I.$ Let $\lambda\in\left[  0,1\right]  .\ $We infer%
\[
\left(  \left(  1-\lambda\right)  f+\lambda g\right)  (y)\geq\left(  \left(
1-\lambda\right)  f+\lambda g\right)  (x)+c\left(  z,y\right)  -c\left(
x,y\right)  ,
\]
therefore $y\in\partial_{c}\left(  \left(  1-\lambda\right)  f+\lambda
g\right)  (x).$
\end{proof}

Our next result can be seen as a counterpart in the framework of
$c-$convexity, for \cite[Lemma 4.1]{GM1996}.

\begin{proposition}
Assume $c:I\times J\rightarrow%
\mathbb{R}
$ is a cost function and{ }$f,g:I\rightarrow%
\mathbb{R}
\ $are{\ }$c-${convex}. Let $X=\left\{  x:f\left(  x\right)  <g\left(
x\right)  \right\}  .$ If there exists $u\in X$ and $v\in I$ such that
\[
\text{ }\partial_{c}g\left(  u\right)  \text{ }\cap\partial_{c}f\left(
v\right)  \neq\varnothing
\]
then $v\in X.$
\end{proposition}

\begin{proof}
Let $y\in\partial_{c}g\left(  u\right)  $ $\cap\partial_{c}f\left(  v\right)
.$ One has
\begin{align*}
g(v)  &  \geq g(u)+c\left(  v,y\right)  -c\left(  u,y\right)  ,\\
f(u)  &  \geq f(v)+c\left(  u,y\right)  -c\left(  v,y\right)  ,
\end{align*}
which implies
\[
g(v)\geq f(v)+\left[  g(u)-f(u)\right]  >f(v).
\]
Hence $v\in X.$
\end{proof}

The remaining results of this subsection were obtained by imposing some
additional conditions to the cost function in order to get a nicer shaped
graph of the set-valued function $\partial_{c}f.$

\begin{proposition}
\label{2affine}{Let} $f:I\rightarrow\mathbb{%
\mathbb{R}
}$ be{ }convex relative to a 2-affine cost function $c$. Then, for all $x\in
I,$ the set $\partial_{c}f(x)$ is convex, possibly empty at the endpoints of
$I.$
\end{proposition}

\begin{proof}
Let $y_{1},$ $y_{2}\in\partial_{c}f(x).$ Then
\[
f(z)\geq f(x)+c\left(  z,y_{i}\right)  -c\left(  x,y_{i}\right)  ,\text{ for
all }z\in I,\text{ }i=1,2.
\]
By direct computation, we obtain
\begin{align*}
f(z)  &  \geq f(x)+\left(  1-\lambda\right)  \left[  c\left(  z,y_{1}\right)
-c\left(  x,y_{1}\right)  \right]  +\lambda\left[  c\left(  z,y_{2}\right)
-c\left(  x,y_{2}\right)  \right] \\
&  =f(x)+\left[  c\left(  z,\left(  1-\lambda\right)  y_{1}+\lambda
y_{2}\right)  -c\left(  x,\left(  1-\lambda\right)  y_{1}+\lambda
y_{2}\right)  \right]  ,
\end{align*}
that is $\left(  1-\lambda\right)  y_{1}+\lambda y_{2}\in\partial_{c}f(x).$
\end{proof}

\begin{remark}
This result represents a counterpart (in the framework of $c-$convexity) of
the assertion that for every convex function $f,$ the sets $\partial f(x)$ are
convex, possibly empty at the endpoints of the domain. It makes sense to us to
denote the upper and lower bounds of $\partial_{c}f(x)$ (if the set is
nonempty and convex) by $f_{-}^{\prime c}\left(  x\right)  ,\ f_{+}^{\prime
c}\left(  x\right)  $ and call them lateral $c-$derivatives$.$
\end{remark}

The set
\[
Y=\left\{  y\in J:\exists x_{1}\neq x_{2}\in I\text{ such that }y\in
\partial_{c}f(x_{1})\cap\partial_{c}f(x_{2})\right\}
\]
has the Lebesgue measure zero (see \cite[Lemma 3.1]{MTW2005}) when $f$ is
lower semicontinuous. Combining this result with Proposition \ref{2affine}, we
derive the following remark.

\begin{remark}
\label{rem_2aff} For a continuous (hence lower semicontinuous) and $c-$convex
function $f$, when dealing with 2-affine costs, the intersections
$\partial_{c}f(x_{1})\cap\partial_{c}f(x_{2}),$ $x_{1}\neq x_{2}\in I$ can
have at most one element. This agrees with the case of usual convex functions.
\end{remark}

\begin{proposition}
Suppose that the cost function $c$ is concave and 2-affine.\ Let
$f:I\rightarrow%
\mathbb{R}
$ be convex and $c-$convex. Then $\partial_{c}f$ is a convex set-valued
function, i.e. for $x_{1},x_{2}\in dom\left(  \partial_{c}f\right)  $ it
holds
\[
\left(  1-\lambda\right)  \partial_{c}f(x_{1})+\lambda\partial_{c}%
f(x_{2})\subset\partial_{c}f(\left(  1-\lambda\right)  x_{1}+\lambda x_{2})
\]
for all $\lambda\in\left[  0,1\right]  .$
\end{proposition}

\begin{proof}
Let $z\in\left(  1-\lambda\right)  \partial_{c}f(x_{1})+\lambda\partial
_{c}f(x_{2})$ for an arbitrary fixed $\lambda\in\left[  0,1\right]  .$ Then we
can write $z=\left(  1-\lambda\right)  a+\lambda b,$ for some $a\in
\partial_{c}f(x_{1})$, $b\in\partial_{c}f(x_{2}).$

Since%
\begin{align*}
f(x)  &  \geq f(x_{1})+c\left(  x,a\right)  -c\left(  x_{1},a\right)  ,\\
f(x)  &  \geq f(x_{2})+c\left(  x,b\right)  -c\left(  x_{2},b\right)  ,
\end{align*}
we get%
\begin{align*}
f(x)  &  \geq\left(  1-\lambda\right)  f(x_{1})+\lambda f(x_{2})+c\left(
x,\left(  1-\lambda\right)  a+\lambda b\right)  -\left(  1-\lambda\right)
c\left(  x_{1},a\right)  -\lambda c\left(  x_{2},b\right) \\
&  \geq f(\left(  1-\lambda\right)  x_{1}+\lambda x_{2})+c\left(  x,\left(
1-\lambda\right)  a+\lambda b\right) \\
&  -c\left(  \left(  1-\lambda\right)  x_{1}+\lambda x_{2},\left(
1-\lambda\right)  a+\lambda b\right)  .
\end{align*}
Therefore%
\[
f(x)\geq f\left(  \left(  1-\lambda\right)  x_{1}+\lambda x_{2}\right)
+c\left(  x,z\right)  -c\left(  \left(  1-\lambda\right)  x_{1}+\lambda
x_{2},z\right)  ,
\]
hence $z\in\partial_{c}f(\left(  1-\lambda\right)  x_{1}+\lambda x_{2}).$

This completes the proof.
\end{proof}

Our next result reads as follows.

\begin{proposition}
\label{ref_an}Let the cost function $c:I\times J\rightarrow%
\mathbb{R}
$ be 1-concave.\ Assume $f:I\rightarrow%
\mathbb{R}
$ is convex and $c-$convex. Then
\begin{equation}
\partial_{c}f(x_{1})\cap\partial_{c}f(x_{2})\subset\partial_{c}f(\left(
1-\lambda\right)  x_{1}+\lambda x_{2}) \label{s_diff_1_in}%
\end{equation}
for all $\lambda\in\left[  0,1\right]  $ and $x_{1},x_{2}\in dom\left(
\partial_{c}f\right)  .$
\end{proposition}

\begin{proof}
We focus on the case $\partial_{c}f(x_{1})\cap\partial_{c}f(x_{2}%
)\neq\varnothing.$ Let $z\in\partial_{c}f(x_{1})\cap\partial_{c}f(x_{2})$.
Then%
\begin{align*}
f(x)  &  \geq f(x_{1})+c\left(  x,z\right)  -c\left(  x_{1},z\right)  ,\\
f(x)  &  \geq f(x_{2})+c\left(  x,z\right)  -c\left(  x_{2},z\right)  ,
\end{align*}
for all $x\in I.$ Let $\lambda\in\left[  0,1\right]  .\ $Consequently%
\begin{align*}
f(x)  &  \geq\left(  1-\lambda\right)  f(x_{1})+\lambda f(x_{2})+c\left(
x,z\right)  -\left(  1-\lambda\right)  c\left(  x_{1},z\right)  -\lambda
c\left(  x_{2},z\right) \\
&  \geq f\left(  \left(  1-\lambda\right)  x_{1}+\lambda x_{2}\right)
+c\left(  x,z\right)  -c\left(  \left(  1-\lambda\right)  x_{1}+\lambda
x_{2},z\right)  ,
\end{align*}
which helps us to deduce $z\in\partial_{c}f(\left(  1-\lambda\right)
x_{1}+\lambda x_{2}).$

Thus the proof is completed.{}
\end{proof}

\begin{example}
The cost function $c\left(  x,y\right)  =-\log\left(  1-xy\right)  ,$ which
appears in the reflector antenna design problem (the far field case
\cite{KW2007}, \cite{wan2004}) is 1-convex. Since the problem deals with
$c-$concave functions, mutatis mutandis Proposition \ref{ref_an} applies.
\end{example}

If we apply Proposition \ref{ref_an} for a cost function which is 1-concave
and 2-affine, we have via Remark \ref{rem_2aff}:

\begin{remark}
For a continuous and $c-$convex function $f$ the set $\partial_{c}f(\left(
1-\lambda\right)  x_{1}+\lambda x_{2})$ has exactly one element for all
$\lambda\in\left(  0,1\right)  ,$ $x_{1}\neq x_{2}\in I\ $such that
$\partial_{c}f(x_{1})\cap\partial_{c}f(x_{2})\neq\varnothing.$ Particularly,
this means when $c\left(  x,y\right)  =xy$ that if there exist two points
$x<y\in I$ such that $f_{+}^{\prime}\left(  x\right)  =f_{-}^{\prime}\left(
y\right)  ,$ then the function is affine on $\left[  x,y\right]  .$
\end{remark}

\begin{corollary}
Let the cost function $c:I\times J\rightarrow%
\mathbb{R}
$ be 1-concave. Assume $f$ is convex on $I$. If there exist $x_{1}<x_{2}$ such
that $\partial_{c}f(x_{1})\cap\partial_{c}f(x_{2})\neq\varnothing,$ then
$\left[  x_{1},x_{2}\right]  \subset dom\left(  \partial_{c}f\right)  $.
\end{corollary}

\begin{proof}
The inclusion (\ref{s_diff_1_in}) still holds and combined with our assumption
yields
\[
\partial_{c}f(\left(  1-\lambda\right)  x_{1}+\lambda x_{2})\neq\varnothing
\]
$\ $for all $\lambda\in\left[  0,1\right]  .$
\end{proof}

\subsection{Local and global $c-$convexity}

Let $I$ be a bounded open interval and $f:I\rightarrow%
\mathbb{R}
.$ We introduce the local $c-$subdifferential by%
\[
\partial_{c}^{l}f(x_{0})=\left\{  y\in J:\exists\varepsilon>0\ \text{such that
}f(x)\geq f(x_{0})+c\left(  x,y\right)  -c\left(  x_{0},y\right)  \text{ for
}x\in U_{\varepsilon}\right\}  .
\]
Here the set $U_{\varepsilon}$ $=\left\{  x:\left\vert x-x_{0}\right\vert
<\varepsilon\right\}  .$ The function $h_{x_{0}}^{l}:U_{\varepsilon
}\rightarrow%
\mathbb{R}
,$
\[
h_{x_{0}}^{l}\left(  x\right)  =f(x_{0})+c\left(  x,y\right)  -c\left(
x_{0},y\right)
\]
is called local $c-$support curve. Note that $\partial f(x_{0})\subseteq
\partial_{c}^{l}f(x_{0}).$

We call a proper function $f:I\rightarrow\left(  -\infty,\infty\right]  $
locally $c$-convex at $x_{0}$ if there exists $\varepsilon>0$\ and
$g:J\rightarrow\left(  -\infty,\infty\right]  $ such that%
\begin{equation}
f\left(  x\right)  =\sup_{y\in\partial_{c}^{l}f(x)}\left\{  c\left(
x,y\right)  -g\left(  y\right)  \right\}  \label{lcc}%
\end{equation}
for $x\in U_{\varepsilon}.$ Then one has
\[
f_{l}^{c}\left(  y\right)  =\sup_{x\in U_{\varepsilon}}\left\{  c\left(
x,y\right)  -f\left(  x\right)  \right\}  \
\]
for all $y\in\partial_{c}^{l}f(x)\ $and%
\[
f_{l}^{cc}\left(  x\right)  =\sup_{y\in\partial_{c}^{l}f(x)}\inf_{z\in
U_{\varepsilon}}\left\{  f\left(  z\right)  +c\left(  x,y\right)  -c\left(
z,y\right)  \right\}  .
\]
Obviously the condition $f=f_{l}^{cc}$ on $U_{\varepsilon}$ is equivalent to
the local $c$-convexity of\ $f$ at $x_{0}.$

\begin{proposition}
Let $f:I\rightarrow%
\mathbb{R}
,$ $\alpha\in I.$ The function $f$ admits a local $c-$support curve at
$\alpha$\ if and only if $f\left(  \alpha\right)  =f_{l}^{cc}\left(
\alpha\right)  .$
\end{proposition}

\begin{proof}
We assume that $f$ admits a local $c-$support curve at $\alpha.$ Let
$y\in\partial_{c}^{l}f(\alpha).$ Then there exists $\varepsilon>0$ such that%
\[
f(z)\geq f(\alpha)+c\left(  z,y\right)  -c\left(  \alpha,y\right)  \ \text{for
all }z\in U_{\varepsilon}.
\]
Thus, since%
\[
\inf_{z\in U_{\varepsilon}}\left\{  f(z)+c\left(  \alpha,y\right)  -c\left(
z,y\right)  \right\}  =f\left(  \alpha\right)  ,
\]
we have
\[
f_{l}^{cc}\left(  \alpha\right)  =\sup_{y\in\partial_{c}^{l}f(\alpha)}%
\inf_{z\in U_{\varepsilon}}h\left(  f(z)+c\left(  \alpha,y\right)  -c\left(
z,y\right)  \right)  =f\left(  \alpha\right)  .
\]

Conversely, let $\varepsilon>0.$ The function $f_{l}^{cc}$ is $c$-convex on
$U_{\varepsilon},$ hence it admits a $c-$support curve at $\alpha,$ that is
there exists $y\in\partial_{c}^{l}f(\alpha)$ such that
\[
f_{l}^{cc}(z)\geq f_{l}^{cc}(\alpha)+c\left(  z,y\right)  -c\left(
\alpha,y\right)  \ \text{for all }z\in U_{\varepsilon}.
\]
Also we know that $f_{l}^{cc}\leq f$ on $U_{\varepsilon},$ which yields
\begin{align*}
f\left(  z\right)   &  \geq f_{l}^{cc}(z)\geq f_{l}^{cc}(\alpha)+c\left(
z,y\right)  -c\left(  \alpha,y\right) \\
&  =f(\alpha)+c\left(  z,y\right)  -c\left(  \alpha,y\right)  \ \text{for all
}z\in U_{\varepsilon}.
\end{align*}
Summarizing the above discussion, there exists $y\in\partial_{c}^{l}f(\alpha)$
such that
\[
f\left(  z\right)  \geq f(\alpha)+c\left(  z,y\right)  -c\left(
\alpha,y\right)  \ \text{for all }z\in U_{\varepsilon}%
\]
and the claim follows.
\end{proof}

\begin{remark}
This agrees with the known fact that the function $f$ admits a supporting line
at $\alpha$ if and only if $f\left(  \alpha\right)  =f^{\ast\ast}\left(
\alpha\right)  $ (see \cite{TB2006}).
\end{remark}

\begin{acknowledgement}
We are very grateful to Dr.\ Eleutherius\ Symeonidis (from
Mathematisch-Geographische Fakult\"{a}t, Katholische Universit\"{a}t
Eichst\"{a}tt-Ingolstadt, Germany) for useful discussions on this paper.
\end{acknowledgement}


\begin{thebibliography}{99}                                                                                               %


\bibitem {caf1996}L. Caffarelli, \emph{Allocation maps with general cost
functions}, Partial differential equations and applications, Lecture Notes in
Pure and Appl. Math., \textbf{177} (1996) 29-35. Dekker, New York.

\bibitem {F1949}W. Fenchel, \emph{On conjugate convex functions}, Canad. J.
Math. \textbf{1} (1949), 73-77.

\bibitem {GM1996}W. Gangbo, R.J. McCann, \emph{The geometry of optimal
transportation, Acta Math.}, \textbf{177}(1996), 113-161.

\bibitem {KW2007}A. Karakhanyan, X.-J. Wang, \emph{The reflector design
problem}, International Congress of Chinese Mathematicians (ICCM) 2007, Vol.
II, 1-4

\bibitem {MTW2005}X.-N. Ma, N.S.\ Trudinger, X.-J. Wang, \emph{Regularity of
Potential Functions of the Optimal Transportation Problem}, Arch. Rational
Mech. Anal. \textbf{177} (2005) 151--183, DOI: 10.1007/s00205-005-0362-9

\bibitem {MN2011}F.-C. Mitroi, Constantin P. Niculescu, \emph{An extension of
Young's inequality}, Abstract and Applied Analysis, Article ID 162049, doi:10.1155/2011/162049

\bibitem {cpn04}C. P. Niculescu, L.-E. Persson, \emph{Old and New on the
Hermite-Hadamard Inequality}, Real Analysis Exchange, (2004), Vol.
\textbf{29}(2), 2003/2004, 663--685.

\bibitem {cpn06}C. P. Niculescu, L.-E. Persson, \emph{Convex Functions and
their Applications. A Contemporary Approach}, CMS Books in Mathematics vol.
\textbf{23}, Springer-Verlag, New York, 2006.

\bibitem {RR1988}S.T. Rachev, L. R\"{u}schendorf, \emph{Mass Transportation
Problems}, Probab. Appl. Springer-Verlag, New York, 1998.

\bibitem {R2007}L. R\"{u}schendorf, \emph{Monge-Kantorovich transportation
problem and optimal couplings}, Jahresber. Deutsch. Math.-Verein., 109 (3)
(2007), 113--137.

\bibitem {TB2006}H. Touchette, C. Beck, \emph{Nonconcave entropies in
multifractals and the thermodynamic formalism}, J. Stat. Phys. 125, 455-471, 2006

\bibitem {TW2009}N. Trudinger, X.-J. Wang, \emph{On strict convexity and
continuous differentiablity of potential functions in optimal transportation,
}Arch. Rational Mech. Anal. \textbf{192} (2009) 403--418, DOI:10.1007/s00205-008-0147-z

\bibitem {TW09}N. Trudinger, X.-J. Wang, \emph{On the second boundary value
problem for Monge-Amp\`{e}re type equations and optimal transportation,} Ann.
Scuola Norm. Sup. Pisa, \textbf{8 }(2009), 1-32.

\bibitem {vil09}C. Villani, \emph{Optimal Transport. Old and New}, Series:
Grundlehren der mathematischen Wissenschaften, Vol. \textbf{338}, 2009

\bibitem {wan2004}X.-J. Wang, \emph{On the design of a reflector antenna II,
}Calc. Var. \textbf{20} (2004), 329--341.
\end{thebibliography}
\end{document}